\documentclass[12pt]{article}
\usepackage{amsmath}
\usepackage{amsfonts}
\usepackage{amssymb}
\usepackage{graphicx,color}%

\providecommand{\U}[1]{\protect\rule{.1in}{.1in}}
\newtheorem{theorem}{Theorem}

\newtheorem{corollary}[theorem]{Corollary}
\newtheorem{definition}[theorem]{Definition}

\newtheorem{lemma}[theorem]{Lemma}
\newtheorem{proposition}[theorem]{Proposition}

\newtheorem{problem}[theorem]{Problem}
\newenvironment{proof}[1][Proof]{\noindent\textbf{#1.} }{\ \rule{0.5em}{0.5em}}

\newcommand{\ab}{\partial_{\infty}}
\newcommand{\arcsec}{\rm{arcsec\,}}
\newcommand{\dive}{{\rm{\,div\,}}}

\newcommand{\Int}{\rm{Int}\,}
\newcommand{\dist}{{\rm{dist}\,}}

\begin{document}
\title{A note on minimal graphs over certain unbounded domains of Hadamard manifolds}
\author{Miriam Telichevesky}
\maketitle

\begin{abstract}
Given an unbounded domain $\Omega$ of a Hadamard manifold $M$, it  makes sense to consider the problem of finding minimal graphs with  prescribed continuous data on its cone-topology-boundary, i.e., on its  ordinary boundary together with its asymptotic boundary. In this  article it is proved that under the hypothesis that the sectional  curvature of $M$ is $\le -1$ this Dirichlet problem is solvable if  $\Omega$ satisfies certain convexity condition at infinity and if  $\partial \Omega$ is mean convex. We also prove that mean convexity of  $\partial \Omega$ is a necessary condition, extending to unbounded  domains some results that are valid on bounded ones.
\end{abstract}

\section{Introduction}\label{introd}

The classical theorem of Jenkins and Serrin on minimal graphs theory, following the works of Bernstein \cite{Bernstein}, Haar \cite{Haar}, Rad\'o \cite{Rado} and Finn \cite{Finn}, states the following.

\begin{theorem}[Theorem 1 of \cite{JS}] \label{thJS} Let $D\subset \mathbb{R}^{n}$ be a bounded domain whose boundary is of class $C^2$. Then the Dirichlet problem for the minimal surface equation in $D$ is well posed for $C^2$ boundary data if and only if the mean curvature of $\partial D$ is everywhere nonnegative.
\end{theorem}

In the last four decades, several works considered problems related to Theorem \ref{thJS} in distinct directions. Some of them are listed below togheter with some references. 
\begin{itemize}
\item Unbounded domains of $\mathbb{R}^2$: \cite{Hw,CK,RoSa,RTo,Kurst, Kuwert, Kutev-Tomi}.
\item Bounded domains of a Hadamard manifold $M$: \cite{FR, MRR, ARS}.
\item Asymptotic Dirichlet problems on Hadamard manifolds: \cite{EFR, RT2, GR,CHR}.
\item Replace the ambient space $\mathbb{R}^{n+1}$ by the hyperbolic spaces $\mathbb{H}^{n+1}$ \cite{BS, GS, L2001b, Nitsche} or other ambient spaces with a Killing field satisfying certain hypotheses \cite{AD,DHL,DLR}. In this setting it is natural to consider CMC Killing graphs and there is an extensive bibliography on it.
\end{itemize}

The purpose of this article is to prove that similar existence and nonexistence results remain valid if in Theorem \ref{thJS} $\mathbb{R}^n$ is replaced by a Hadamard manifold $M$ with sectional curvature $K_M \le -1$ and the domain $D$ is unbounded and ``strictly convex at infinity'', see Definition \ref{defsc}.

Classically Dirichlet problems on unbounded domains are considered in $\mathbb{R}^n$ without prescribed values at infinity. In fact, sometimes the behavior at infinity of bounded solutions are determined by their boundary values. For instance, in $\mathbb{R}^2$ it is a consequence of Theorem 2 of \cite{CK}, that states that if $u$ and $v$ are distinct solutions of the Dirichlet problem in an unbounded domain $U\subset \mathbb{R}^2$ that coincide on $\partial U$, then $\sup |u-v|$ must have at least logarithm growth. However, since the manifolds that we consider in this work have sectional curvature $K_M \le -1$, it turns out that the asymptotic boundary of unbounded domains may be ``good enough'' to prescribe continuous data on them. It therefore makes sense to consider the generalized Dirichlet problem for the minimal hypersurface equation, Problem \ref{dpr}, described in the sequel. In order to state it, let us introduce some useful notations that are not standard.

Throughout this paper $M$ denotes a $m$-dimensional Hadamard manifold, $m\ge 2$, with sectional curvature $K_M$ satisfying $K_M \le -1$. The asymptotic boundary $\ab M$ of $M$ is defined by the set of equivalence classes of geodesic rays that stay at finite distance for all time and it is possible to compactify $M$ adding $\ab M$ to it; $\overline{M}:=M\cup \ab M$ carries the so-called cone topology, see \cite{EO}, which makes it canonically homeomorphic to a closed ball. If $U\subset \overline{M}$ is any set, we denote by $\overline{U}^{ct}\subset\overline{M}$ and $\partial ^{ct}U\subset \overline{M}$ its closure and boundary in terms of cone topology; we also use the notation $\ab U:=\partial^{ct} U \cap \ab M$.

\begin{problem}\label{dpr}
Let $\Omega\subset M$ be a $C^2$ domain of $M$. Given $\varphi \in {C}(\partial^{ct} \Omega)$, find a minimal graph over $\Omega$ that atains $\varphi$ on its boundary, or, equivalently, to find a solution of the Dirichlet problem
$$
\left\{\begin{array}{l}u\in {C}^2(\Omega)\cap {C}\left(\overline{\Omega}^{ct}\right),\\\mathcal{M}(u):=\dive\left(\displaystyle\frac{\nabla u}{\sqrt{1+|\nabla u|^2}}\right) = 0 \text{ in }\Omega,\\ u|_{\partial^{ct} \Omega} = \varphi.\end{array}\right.
$$
\end{problem}

\

Concerning the existence part, maybe the main difficulty to deal with unbounded domains is the nonexistence of natural barriers. In general, barriers are constructed using distance function to a point or to the boundary of the domain, which is impossible to adapt directly to points at infinity. The geometry of $M$ at infinity plays an important role at this point. For instance, the hyperbolic spaces $\mathbb{H}^n$ have ``good geometry'' at infinity by the existence of hyperspheres separating points at infinity and having their principal curvatures with the correct sign. The natural way to generalize this fact in order to use Hessian's comparison Theorem and adapt barriers of $\mathbb{H}^n$ to another Hadamard manifolds is given by the {\em strict convexity condition (SC condition) at infinity}, introduced in \cite{RT2}. In that work it is proved that Problem \ref{dpr} is solvable for $\Omega = M$ (in this case, it is called the Asymptotic Dirichlet Problem) and for any continuous boundary data if $M$ satisfy the SC condition, described below.

\begin{definition}\label{defsci}
A Hadamard manifold $M$ is said to satisfy the strict convexity condition at infinity if for all $x\in \ab M$ and all relatively open subset $W\subset \ab M$, there exists $\Omega \subset M$ of class $C^2$ such that $M\backslash \Omega$ is convex.
\end{definition}

At this point, it should be mentioned that under the hypothesis $K_M \le -1$, the SC condition is always satisfied by $2$-dimensional manifolds, by the rotationally symmetric ones and by those manifolds with controled decay on sectional curvature (exponential decay), see also \cite{RT2}. We also should mention that under the same assumption on $K_M$, the SC condition is equivalent to the {\em convex conic neighborhood condition} presented by H. Choi in \cite{Choi} in the study of the asymptotic Dirichlet problem with respect to Laplace's operator on Cartan-Hadamard manifolds; the equivalence is a consequence of Lemma 1 of the work \cite{B98-2} of A. Borb\'ely. In fact, both Dirichlet problems are closely related and may be studied together, see also \cite{RT2}.

Contrasting with the existence results under SC condition, we cite the counterexample constructed in \cite{HR} by I. Holopainen and J. Ripoll. In this work the authors present a Hadamard manifold with $K_M \le -1$ that does not admit solution to the asymptotic Dirichlet problem for minimal hypersurface equation for any continuous $\varphi\in C(\ab M)$, although there are bounded non-constant minimal graphs globally defined on $M$ , see Theorem 1.1 of \cite{HR}. This counterexample proves that the condition $K_M \le -1$ is not sufficient to solve Problem \ref{dpr} with any continuous boundary data. 

Taking account all these facts, the following definition is natural.

\begin{definition}\label{defsc}
A domain $\Omega\subset \overline{M}$ is called {\em strictly convex at infinity} if for any $x\in \ab \Omega$ and any $\Gamma\subset \partial^{ct} \Omega$ relatively open neighborhood of $x$ there exists $V=V(x,\Omega)\subset \Omega$ open neighborhood of $x$ such that $\overline{V}\cap \partial \Omega \subset \Gamma$ and all the principal curvatures of $\partial V \cap \Omega$, oriented in the direction of $\Omega \backslash V$, are non-negative.
\end{definition}

We notice that when $\Omega=M$, this definition coincides with SC condition.

\

With Definition \ref{defsc} it is possible to state our main existence result.

\begin{theorem}\label{thmexistence}
Let $\Omega \subset M$ be a mean-convex domain (with respect to the inward orientation) that is strictly convex at infinity. Then Problem \ref{dpr} is solvable for any continuous boundary data.
\end{theorem}

Returning our attention to Theorem \ref{thJS}, when $\Omega\subset \mathbb{R}^n$ is bounded, the mean-convexity is a necessary condition to the solvability of Problem \ref{dpr} for any continuous $\varphi$. The second part of this article is dedicated to prove that mean convexity is also necessary in $M$ if we deal with unbounded domains and require boundedness of solutions. In Section \ref{sectionnonexistence} we present some necessary lemmata and the proof of the following nonexistence result.	

\begin{theorem}\label{thmnonexistence}
Let $\Omega\subset \overline{M}$ be a domain and suppose that there exist $y\in \partial \Omega$ such that the mean curvature of $\partial \Omega$ at $y$ (w.r.t. the inward orientation) satisfies $H(y)<0$. Then there exists a continuous function $\varphi:\partial^{ct} \Omega \to \mathbb{R}$ such that Problem \ref{dpr} is not solvable.
\end{theorem}

The construction of $\varphi$ needs basically two ingredients: First of all, the local aspect concerning about the negativity of the mean curvature $H(y)$. It is essential to guarantee that $\varphi(y)$ is bounded by values of the solution on a small sphere centered at $y$, say, on $S_r(y)\cap \Omega$. The second essential ingredient is the existence of a bounded barrier in  $\Omega\setminus B_r(y)$ with some special properties. Similar results outside $\mathbb{R}^n$ were proved on bounded domains considering barriers depending on the diameter of $\Omega$, as in \cite{Nitsche}. Our main improvement is to drop the dependence on the size of the domain.

Combining the results of Theorems \ref{thmexistence} and \ref{thmnonexistence}, we get 

\begin{theorem}\label{iff}
Let $\Omega \subset M$ be a domain that is strictly convex at infinity. Then the Dirichlet problem \ref{dpr} is solvable for any continuous boundary data if and only if $\Omega$ is mean convex.
\end{theorem}

\

It remains an open question what happens if we assume that $\Omega$ is not strictly convex at infinity. We conjecture that in this case it is also possible to construct a continuous function on $\partial^{ct} \Omega$ for which the Dirichlet problem is not solvable and therefore strict convexity at infinity is also a necessary condition. Since it deals with non existence of solutions in arbitrarly large domains, Theorem \ref{thmnonexistence} may have an important role on the study of this conjecture. 

\

To finish, we should mention that there is a big gap between the behavior of $K_M$ at infinity on the cases where Theorem \ref{thmexistence} is true and on the ones where it is false. It also remains unknown if there exist some sharp condition on $K_M$ that assures solvability of Problem \ref{dpr} for any continuous boundary data.

\section{Existence result}\label{existence}

This section is dedicated to prove Theorem \ref{thmexistence}. We start with a very important tool, the comparison principle for unbounded domains. It plays an important role in both existence and uniqueness parts%, since it guarantees that supersolutions work as upper barriers. 
.
At this moment we just need to work with functions that extend continuously to the asymptotic boundary, however in Section \ref{sectionnonexistence} we treat with a larger class of functions, as stated above.

\begin{proposition}[Comparison Principle for unbounded domains]
Let $U\subset M$ be an un\-boun\-ded domain of $M$. If $u,v\in C^2(U)$ are such that $\mathcal{M}(v) \le \mathcal{M}(u)$ on $U$ with $\limsup_{p\to x} u\le \liminf_{p\to x} v$ for all $x\in\partial^{ct} U$, then $u\le v$ in $U$.\label{comparison}  
\end{proposition}

\begin{proof}
Choose $o\in M$. Let $\varepsilon>0$ be an arbitrary constant. Using the basis of the cone topology of $\overline{M}$, we obtain that for all $x\in \ab  U$, there is an open truncated cone $N_x:= T_o(x,\alpha_x,R_x)$ (that is, the image of a truncated cone of opening angle $\alpha_x$ and radius $R_x$ by the exponential map of some point $o$) such that $u<v+\varepsilon$ on $N_x$. Since $ \ab U$ is compact, there exists uniform $R$ such that $u<v+\varepsilon$ on $U\backslash B_R(o)$. In addition, notice that the hypothesis implies that $u\le v$ on $\partial U $. Therefore we have $u\le v+\varepsilon$ on $\partial (U\cap B_R(o))$, which imply, by the Comparison Principle on bounded domains, that $u\le v+\varepsilon $ on $U\cap B_R(o)$, and hence the last inequality holds on $U$. Since $\varepsilon$ is arbitrary, the proof is complete.
\end{proof}

%The proof is done on two steps. In the first one, we solve the Dirichlet problem in an exhaustion of $\Omega$ by compact domains with boundary data close to $\varphi$ and then extract a uniformly convergent on compacts subsets subsequence to obtain a solution $u:\Omega \to \mathbb{R}$. It then remains to show that it extends continuously to $\partial^{ct} \Omega$, what is done on the second part of the proof, with the construction of local barriers at infinity. 

In the sequel, we prove Theorem \ref{thmexistence}.

\

We first claim that if Problem \ref{dpr} admits solution for all boundary data in $C^{\infty}(\partial \Omega)\cap C(\partial^{ct} \Omega),$ then it admits solution if $\varphi \in C(\partial^{ct}\Omega)$. Indeed, we consider increasing and decreasing sequences $\{\varphi_n^-\}_{n}, \, \{\varphi_n^+\}_{n} \subset C^{\infty}(\partial \Omega)\cap C^0(\partial^{ct}\Omega)$,  respectively, converging to $\varphi$. By assumption, for each $n$ we obtain $u_n^{-},u_n^{+}\in C^2(\Omega)$ solutions of $\mathcal{M}=0$ in $\Omega$ and with $u_n^-=\varphi_n^-$, $u_n^{+}=\varphi_n^+$ on $\partial^{ct} \Omega$. By the Comparison Principle $\{u_n^-\}_n$ is increasing and $\{u_n^+\}_n$ is deacreasing with $n$. Interior gradient estimates of \cite{Spruck} guarantee that we may extract converging on compacts subsets subsequences of them and then obtain $C^2$ functions $u^-\le u^+$ ,solutions of $\mathcal{M}=0$ in $\Omega$, with boundary values given by $\varphi$. Comparison Principle implies that $u^-=u^+$: that is the desired $u$.

From now on, we will assume w.l.g. that $\varphi\in C^{\infty}(\partial \Omega) \cap C(\partial ^{ct}\Omega).$

Let $\Omega_1 \subset \Omega_2 \subset \dots$ be an exhaustion of $\Omega$ by bounded mean-convex domains. Theorem 2 of \cite{DHL}  (with $H=0$, and hence with any $r_0$) implies that there exist $u_k\in C^2(\Omega_k)\cap C(\overline{\Omega_k})$ such that $\mathcal{M}(u_k)=0$ in $\Omega_k$ and $u_k|_{\partial \Omega_k}=\varphi$. Applying Theorem 1.1 of \cite{Spruck} and the diagonal method we obtain a subsequence of $u_k$ uniformly convergent on compacts of $\Omega$, converging to a function $u\in C^{2}(\Omega)$ that has minimal graph. 

Since $\partial \Omega$ mean convex, standard arguments guarantee that the solution assumes the desired data on $\partial \Omega$. To conclude the proof one needs to guarantee that it also extends continuously to $\ab \Omega$, hence in the following we construct barriers at infinity.

\

  We say that a function $\Sigma\in C^{0}\left(\Omega\right)$ is a {\em supersolution for $\mathcal{M}$} if, given a bounded subdomain $U\subset \Omega$, if $u\in C^2(U)\cap C^{0}\left(\overline{U}\right)$ is a solution of
$\mathcal{M}=0$ in $U$, the condition $u|_{\partial U}\leq\Sigma|_{\partial U}$ implies that $u\leq\Sigma|_{U}$. We may also define subsolution replacing $\le$ by $\ge$.

  Given $x\in\partial^{ct} \Omega$ and an open subset $V$ such that $x\in\partial^{ct} V \cap \partial^{ct} \Omega$, we call an {\em upper barrier for $\mathcal{M}$ relative to $x$ and $V$ with height $C$} a function $\Sigma\in C(\Omega)$ such that

\begin{enumerate}
\item[(i)] $\Sigma$ is a supersolution for $\mathcal{M}$;
\item[(ii)] $\Sigma\ge 0$ and $\displaystyle\lim_{p\in M,\,p\to x} \Sigma(p) =0$, the limit with respect to the cone topology;
\item [(iii)] $\Sigma_{\Omega\setminus V} \ge C$.
\end{enumerate}
Lower barriers are defined analogously.

\

  A point $x\in \partial^{ct} \Omega$ is said to be regular (w.r.t. the mean curvature operator $\mathcal{M}$) if it  satisfies the following property:  given $C>0$ and a relatively open subset $W\subset\overline{\Omega}^{ct}$ with $x\in W$, there exist an open set $V\subset W$ such that $x\in\operatorname*{Int} \partial^{ct} V\subset \partial^{ct} W$ and an upper barrier $\Sigma:\Omega\rightarrow\mathbb{R}$ relative to $x$ and $V$, with height $C$. 

\

\begin{lemma}
 The function $u$ obtained by the exhaustion argument extends continuously to $\varphi$ to each regular point $x\in \ab \Omega$. 
\end{lemma}

 \begin{proof} Let $x\in\ab \Omega$ and $\varepsilon>0$. Since $\varphi$ is continuous, there exists an open neighborhood $W\subset \overline{\Omega}^{ct}$ of $x$ such that $\varphi(y) < \varphi(x)+\varepsilon/2$ for all $y\in W$. Furthermore, the regularity of $x$ implies that there exists an open subset $V\subset W$ such that $x\in\operatorname*{Int}\partial^{ct} V \subset \partial^{ct} W$ and an upper barrier $\Sigma: \Omega \to\mathbb{R}$ with respect to $x$ and $V$ with
height $C:=\max_{\overline{\Omega}} |\varphi|$. 

 Defining $$v(q):=\Sigma(q)+\varphi(x)+\varepsilon,$$ we claim that $u\le v$ in $V$:  In fact, let $V_{k}:=V\cap\Omega_{k}$, $k\ge k_{0}$, where $k_{0}$ is choosen such that $\Omega_{k_{0}}\cap V \neq\emptyset$. Notice that $u_{k}\le v$ in $V_{k}$ because the inequality holds on $\partial V_{k}=\overline{\left(\partial\Omega_{k}\cap V\right) }\cup\overline{\left( \partial V\cap\Omega_{k}\right) }$: on $\partial\Omega_{k} \cap V$, it is true due to the choice of $V$; on $\partial V\cap\Omega_{k}$, it holds because $\Sigma\ge\max|\varphi|$ on $\partial V\cap \Omega$, which implies that $\Sigma\ge u_{k}$, by the Comparison Principle.

Since $u_k\le v$ for all $k\ge k_{0}$, we have $u\le v$ on $V$.

 It is also possible to define $v_{-}:M \to\mathbb{R}$ by $v_{-}(q):=\varphi (x)-\varepsilon- \Sigma(q)$ in order to obtain $u\ge v_{-}$ in $V$. It then holds that $$|u(q) - \varphi(x) | < \varepsilon+ \Sigma(q),\,\forall\, q\in V,$$ and hence $$\limsup_{p\to x} |u(p) - \varphi(x)| \le\varepsilon.$$
The proof is then complete, since $\varepsilon>0$ is arbitrary.
\end{proof}

 To finish, we are now going to prove regularity at the points of $\ab \Omega$. 

\begin{proposition}\label{barrier}
Let $\Omega \subset M$ be a domain that is strictly convex at infinity. Then $\mathcal{M}$ is regular at each point of $\ab \Omega$. 
\end{proposition}

\begin{proof}
Let $x\in \ab \Omega$ and let $W\subset \overline{\Omega}^{ct}$ be an open neighborhood of $x$. By hypothesis, there exist $V\subset \Omega$ open neighborhood of $x$ such that $\overline{V}\cap \partial^{ct} \Omega \subset \partial^{ct} W=:\Gamma$ and $\partial V \cap \Omega$ has nonnegative principal curvatures.  We may assume without loss of generality that $V\subset W$: by Hessian's comparison Theorem, all equidistant hypersurfaces of $\partial V\cap \Omega$ contained in $V$ have nonnegative principal curvature  (oriented to $\Omega\setminus V$); furthermore, since $\partial^{ct}V\subset \Gamma$, it holds that after some finite distance all hypersurfaces equidistant to $\partial V \cap \Omega$ are contained in $W$. 

\
 
Consider $s: V \to \mathbb{R}$ the distance function to $\partial V \cap \Omega$. Since $K_M \le -1$ and all principal curvatures of $\partial V \cap \Omega$ are nonnegative, we have that the Laplacian of $s$ satisfies $\Delta s \ge (n-1)\tanh s$ (see, for instance, Theorem 4.3 of \cite{Choi}). 

 Define $g: (0,+\infty) \to \mathbb{R}$ by 
$$g(s) := \displaystyle\int_s^{+\infty}{\frac{dt}{\sqrt{\cosh^{2(n-1)}t - 1}}}.$$ Notice that $g$ is well-defined and $\lim_{s\to 0^+} g(s) = +\infty$, $\lim_{s\to + \infty} g(s) = 0$. Define now $w: V \to \mathbb{R}$ by $w(p):=g(s(p))$. A straightfoward computation gives

$$\mathcal{M}(w) = (n-1)\cosh^{n-1}s \sinh s + (1-n)\cosh^{1-n}s \Delta s$$ and hence the estimative $\Delta s \ge \tanh s$ leads to $\mathcal{M}(w) \le 0$. 

We remark that $w$ is a solution if $M=\mathbb{H}^n$ and $V$ is a totally geodesic hypersphere.

 To finish with the proof, define the supersolution $\Sigma\in C^{0}\left(\overline{\Omega}\right)$ by
$$ \Sigma(p)=\left\{ \begin{array}[c]{ll} \min\left\{w(p),C\right\}  & \text{if }p\in V\\
C & \text{if }p\in\overline{\Omega}\setminus V, \end{array}\right.$$ 
which is of course an upper barrier relative to $x$ and $V$ with height $C$ and hence the proof is complete.
\end{proof}

\section{Nonexistence result}\label{sectionnonexistence}

We now prove that mean-convexity of $\partial \Omega$ is a necessary condition, as stated  in Theorem \ref{thmnonexistence}. 

 We start with the next classical Lemma, proved by Jenkins and Serrin on \cite{JS} in the case where the domain is bounded and contained on $\mathbb{R}^n$.
\begin{lemma}\label{lemmaJS}
Let $U\subset M$ be an open domain and $\Gamma$ relatively $C^1$ open subset of $\partial U$. If $u\in C(\overline{U})\cap C^2(U\cup \Gamma)$ and $v\in C(\overline{U})\cap C^2(U)$ satisfy 
\begin{eqnarray}
\mathcal{M}(v)<\mathcal{M}(u)\text{ in } U;\\ u\le v \text{ on }\partial U \backslash \Gamma\text{ and } \\ \frac{\partial v}{\partial \nu}=-\infty\text{ in } \Gamma, 
\end{eqnarray}
where $\nu$ is the inner normal vector to $\partial U$, then $u\le v$ in $U$.
\end{lemma}
\begin{proof}
If $u\le v$ on $\Gamma$, the result is a consequence of Comparison Principle. 
Suppose by contradiction that there exists $y\in \Int \Gamma$ such that $d:=\max_{\Gamma} (u-v) = u(y)-v(y)>0$. Then $u\le v+d$ on $\partial U$, and hence by the Comparison Principle we have $u\le v+d$ in $U$. Therefore
$$\frac{\partial}{\partial \nu}(u-v)(y) \le 0 \Rightarrow \frac{\partial}{\partial \nu}(u)(y) \le -\infty$$ contradicting the fact that $u\in C^2(U\cup \Gamma)$.
\end{proof}

\begin{lemma}\label{lema1}
Let $\Omega\subset M$ be an open $C^2$ domain (possibly unbounded) with mean curvature (with respect to the inner normal) $H:\partial \Omega\to \mathbb{R}$. Suppose that there exists $y\in \partial\Omega$ such that $H(y)<0$. Then there exists $s>0$ depending only on the local geometry of $\Omega$ near $y$ and $C>0$ depending only on $H(y)$ such that if $u\in C^2(\overline{\Omega})\cap C(\overline{\Omega}^{ct})$ satisfy $\mathcal{M}(u)=0$ in $\Omega$, then 
$$u(y) \le C + \sup_{\partial B_s(y)\cap \Omega} u.$$ 
\end{lemma}
\begin{proof}
Let $d:\widetilde{\Omega}\to \mathbb{R}$ be given by $d(x)=\dist(x,\partial \Omega)$, where $\widetilde{\Omega}\subset \Omega$ is the open subset where $d$ is smooth. 

 Since $H(y)<0$, it holds that $\Delta d(y)=-H(y)>0$. Since $\partial \Omega$ is $C^2$, there exists $s>0$ such that $B_s(y)\cap \Omega \subset \widetilde{\Omega}$ and $$\Delta d(x)>-\frac{H(y)}{2}=:\epsilon,\,\,\forall x\in B_s(y)\cap \Omega.$$ This is the required $s$. 

 We claim that if $x\in B_s(y)\cap \Omega$, then $u(x) \le C + \sup_{\partial B_s(y)\cap \Omega} u$. To prove it, let $\Gamma_x$ be the level set of $d$ that contains $x$ and $\Omega_x$ be the set enclosed by $\Gamma_x$ and $\partial B_s(y)$, that is, $\Omega_x:=\{p\in B_s(y)\,|\,d(p)>d(x)\}.$

\

 Consider $\psi$ given by 

\begin{equation}\label{psi}
\psi(t)= \frac{\pi}{2} - \arcsec(t+1).
\end{equation}
 Then $\psi\ge 0$, $\psi(0)=\pi/2$ and $\lim_{t\to + \infty}\psi(t) = 0$. Its first and second derivatives are given below:
$$\psi'(t)= - \frac{1}{(t+1)\sqrt{t^2+2t}};$$
$$\psi''(t) = \frac{1}{(t+1)^2\sqrt{t^2+2t}} + \frac{1}{(t^2 + 2t)^{3/2}}.$$
 Define $w: B_s(y)\cap \Omega_x \to \mathbb{R}$ by 
$$w(p):=A \psi(d(p)) + \sup_{\partial B_s(y)\cap \Omega} u,$$ where $A>0$ is to be determined.
 After some computations we obtain

$$(1+|\nabla w|^2)^{3/2}\mathcal{M}(w)(p) = A \psi''(d(p)) + (A \psi'(d(p)) + A^3\psi'(d(p))^3)\Delta d(p).$$
 Using then that $\Delta d(p)>\epsilon$ and $\psi'<0$ in the domain we are considering, we obtain
\begin{eqnarray*}
(1+|\nabla w|^2)^{3/2}\mathcal{M}(w) &\le& A \left[\psi'' + \epsilon\psi'+\epsilon A^2 \psi'^3\right]\\
&=& A (t+1)^{-3}(t^2+2t)^{-3/2}\left[(t+1)(t^2+2t) + (t+1)^3 \right.\\&&-\left. \epsilon(t+1)^2(t^2+2t) - \epsilon A^2\right].
\end{eqnarray*}
 Notice that the term into the brackets is a polynomial of degree $4$, leader coefficient $-\epsilon<0$ and independent term $1-\epsilon A^2$. Then it is clear that there exists $A>0$ large enough in such a way that this polynomial is negative for all $t\ge 0$; with this choice of $A$ we obtain that $\mathcal{M}(w)<0$ on $\Omega_x$. 

 Furthermore, by definition of $w$ we have $w\ge u$ on $\partial B_s(y)\cap \Omega_x$ and $\partial w/\partial \nu=-\infty$ on $\Gamma_x$, which is an open $C^1$ portion of $\partial \Omega_x$. We also notice that $u\in C^2(\Gamma_x)$. By Lemma \ref{lemmaJS}, we obtain $w\ge u$ in $\Omega_x$. Since $x$ is arbitrary and $u$ is continuous, it holds the desired inequality with $C= A\frac{\pi}{2}$ what concludes the proof. 
\end{proof}

\begin{proposition}\label{prop2}
Let $M$ be a Hadamard manifold with sectional curvature $K_M \le -1$. There exists universal $C>0$ such that if $\Omega$ is a $C^1$ domain of $M$ and $u$ satisfy $\mathcal{M}(u)=0$ in $\Omega$, then
$$\sup_{\partial B_s(y) \cap \Omega} u \le C + \sup_{\partial^{ct} \Omega \backslash B_s(y)} u$$ for all $y\in \partial \Omega$ and $s>0$ such that $\partial B_s(y) \cap \Omega$ is a nonempty connected set.
\end{proposition}
\begin{proof}
Consider $v:\Omega \backslash B_s(y) \to \mathbb{R}$ given by $$v(x)=B \psi(r(x)) + \sup_{\partial^{ct} \Omega \backslash B_s(y)} u,$$ where $\psi$ is given by \eqref{psi}, $r(x):= \dist(x,\partial B_s(y))$ and $B$ is an appropriated constant to be chosen latter. Since $K_M\le -1$, we have by the Hessian's Comparison Theorem that $\Delta r \ge n-1$. Hence mimicking the computations of the previous Lemma we obtain the same polyomial, except that we have $n-1$ instead of $\epsilon$ and $B$ instead of $A$. It is again clear that there exists $B$ large enough such that $\mathcal{M}(v)\le 0$. We remark that such constant does not depend on anything (except in the fact that $K_M \le -1$) since we may choose the constant that is appropriate to the case $n=2$ and it works on all dimensions.

 We are again on the hypotheses of Lemma \ref{lemmaJS}, now considering $U=\Omega \backslash B_s(y)$. Hence we obtain, for all $x\in \partial B_s(y) \cap \Omega$, 
$$u(x) \le \sup_{\partial^{ct} \Omega\backslash B_s(y)} u + B\frac{\pi}{2}$$ and the proof is complete.
\end{proof}

\

\begin{proof}[Proof of Theorem \ref{thmnonexistence}]
Combining the estimates obtained on Lemma \ref{lema1} and Proposition \ref{prop2}, we obtain the existence of a continuous function $\varphi:\partial \Omega \to \mathbb{R}$ for which Problem \ref{dpr} is not solvable: it suffices to put $\varphi(y)=\pi(A + B)$, where $A$ and $B$ are given by the previous results, and $\varphi = 0$ on $\partial \Omega \setminus B_s(y)$, where $s$ is given on the proof of Lemma \ref{lema1}. 
\end{proof}

\section{Applications}

\begin{corollary}\label{isolated} 
Let $\Omega$ be a domain that has only finitely many points on $\ab \Omega$. Then Problem \ref{dpr} is solvable for any continuous $\varphi$ if and only if $\Omega$ is mean-convex. 
\end{corollary}
\begin{proof}
Notice that since $\ab M$ is compact, then $\ab \Omega$ is also compact and therefore ``finitely many points on $\ab \Omega$'' is equivalent to ``isolated points on $\ab \Omega$''. In order to apply Theorem \ref{iff}, it suffices to prove that $\Omega$ is strictly convex at infinity.

 Let $x\in \ab \Omega$ and $W\subset \overline{\Omega}^{ct}$ a relatively open neighborhood of $x$. We may suppose without loss of generality that $x$ is the only point at infinity of $W$, otherwise we just work with any open subset of $W$ where this property holds. Choosen $o\in M\backslash W$, we have that $W$ is contained on some truncated cone centered at $o$ with radius $R>0$, and as a consequence we have $\partial W \subset  M\setminus B_R(o)$. Set $V:=\Omega\backslash \overline{B_R(o)}$, and it is clear that it satisfies the required conditions.	
\end{proof}

\begin{corollary}
If $M$ satisfies the SC condition and $\Omega$ is a mean-convex domain of $M$ such that $\ab \Omega$ is composed only by open portions and isolated points, then Problem \ref{dpr} is solvable in $\Omega$. In particular, this is the case where if either $\dim M =2$, or $M$ is rotationally symmetric, or 
\begin{equation}
\label{aRcondition}\min\{K_{M}(\Pi);\text{ }\Pi\text{ is a }2-\text{plane in }T_{p}M,\text{ }p\in B_{R+1}(o)\} \geq-\frac{e^{2kR}}{R^{2+2\epsilon}},\,\,R\geq
R^{\ast},
\end{equation}
for some constants $\epsilon,$ $R^{\ast}>0$.
\end{corollary}
The particular cases of Corollary above may be found in \cite{RT2}.

\subsection{Application of the technique: Dirichlet problems for $p$-Laplacians}\label{plaplace}

Consider now the following Dirichlet problem for the $p$-Laplacian operator, $p>1$, for continuous $u$ in the Sobolev space $W^{1,p}(\Omega)$:

\begin{equation}\label{dplap}
\left\{\begin{array}{l}\Delta_p(u):=\dive\left(|\nabla u|^{p-2}\nabla u\right) = 0 \text{ in }\Omega,\\ u|_{\partial \Omega} = \varphi.\end{array}\right.
\end{equation}

Concerning the case $\Omega=M$, the counterexamples constructed by A. Ancona in \cite{Anc} and by A. Borb\'ely in \cite{B2} show that some convexity at infinity is also needed to obtain existence of solutions of asymptotic Dirichlet problems related to the Laplacian operator $\Delta$. In \cite{Ilkka}, I. Holopainen constructed a counterexample for the $p$-Laplacian operator $\Delta_p$. The manifolds constructed by them contain a point in $\ab M$ with the property that any open neighborhood of it has the whole manifold as the convex hull, and hence $M$ is not strictly convex at infinity. 

On the other hand, in \cite{RT2} the authors proved that the SC condition is sufficient for solvability of asymptotic Dirichlet problems w.r.t. $\Delta_p$. We may extend this result to our case, proving that if $\Omega$ is stricly convex at the infinity, then all $x\in\ab \Omega$ is regular with respect to the operator $\Delta_p$. 

The proof is {\em mutatis mutandis} the same as we have done above; it is sufficient to replace $\mathcal{M}$ by $\Delta_p$ and the function $g$ constructed in Theorem \ref{barrier} by the following one:
$$g(s):=c\int_s^{+\infty} \cosh^{\frac{1-n}{p-1}}(t)dt,$$
where $c$ is a sufficiently large constant ($c=2C(\cosh 1 )^{\frac{n-1}{p-1}}$ works good).

Together with the classical theory of existence of solution over bounded domains that satisfy the exterior sphere condition, we obtain the following result.

\begin{theorem}
Let $M$ be a Hadamard manifold with sectional $K_M \le -1$. Let $\Omega\subset M$ be an unbounded domain that is strictly convex at infinity and that satisfies the exterior sphere condition on its finite part, namely, given $x\in\partial \Omega$, there exist a sphere contained in $M\setminus\Omega$ that is tangent to $\partial \Omega$ at $x$. Then given Problem \eqref{dplap} is solvable for any $\varphi \in C(\partial^{ct} \Omega)$.
\end{theorem}

\end{document}